\documentclass[12pt, reqno, a4paper]{amsart}

\usepackage{amssymb}
\usepackage{hyperref}
\usepackage[matrix, arrow, curve]{xy}

\usepackage[english]{babel}
\usepackage{ulem}

\theoremstyle{plain}
\newtheorem{theorem}{Theorem}[section]
\newtheorem{lemma}[theorem]{Lemma}
\newtheorem{proposition}[theorem]{Proposition}
\newtheorem{corollary}[theorem]{Corollary}

\theoremstyle{remark}
\newtheorem{remark}[theorem]{Remark}

\theoremstyle{definition}
\newtheorem{example}[theorem]{Example}
\newtheorem{problem}[theorem]{Problem}
\newtheorem{definition}[theorem]{Definition}
\newtheorem*{notation}{Notation}
\newtheorem*{question}{Question}

\numberwithin{equation}{section}

\DeclareMathOperator{\ch}{char}

\DeclareMathOperator{\End}{End}

\DeclareMathOperator{\supp}{supp}
\DeclareMathOperator{\diff}{diff}

\DeclareMathOperator{\Hom}{Hom}

\begin{document}

\title[On weak equivalences of gradings]{On weak equivalences of gradings}

\author{Alexey Gordienko}
\address{Vrije Universiteit Brussel, Belgium}
\email{alexey.gordienko@vub.ac.be}

\author{Ofir Schnabel}
\address{University of Haifa, Israel}
\email{os2519@yahoo.com}


\begin{abstract}
When one studies the structure (e.g. graded ideals, graded subspaces, radicals,\dots) or graded polynomial identities of graded algebras, the grading group itself does not play an important role, but can be replaced by any other group that realizes the same grading. Here we come to the notion of weak equivalence of gradings: two gradings are weakly equivalent if there exists an isomorphism between the graded algebras that maps each graded component onto a graded component. 
The following question arises naturally: when a group grading on a finite dimensional algebra is weakly equivalent to a grading by a finite group?
It turns out that this question can be reformulated purely group theoretically in terms of the universal group of the grading.
 Namely, a grading is weakly equivalent to a grading by a finite group if and only if the universal group of the grading is residually finite with respect to a special subset of the grading group. The same is true for all the coarsenings of the grading if
and only if the universal group of the grading is hereditarily residually finite with respect to the same subset.
We show that if $n\geqslant 349$, then on the full matrix algebra $M_n(F)$ there exists an elementary  group grading that is not weakly equivalent to any grading by a finite (semi)group, and if $n\leqslant 3$,
then any elementary grading on $M_n(F)$ is weakly equivalent to an elementary grading by a finite group.
\end{abstract}


\thanks{The first author is supported by Fonds Wetenschappelijk Onderzoek~--- Vlaanderen post doctoral fellowship (Belgium). The second author was partially supported by ISF grant 797/14.}

\maketitle

\section{Introduction}
When studying graded algebras, one has to determine, when two
graded algebras are considered ``the same'' or equivalent.

Recall that a decomposition $\Gamma \colon A=\bigoplus_{s \in S} A^{(s)}$
of an algebra $A$ over a field $F$ into a direct sum of subspaces $A^{(s)}$
is a \textit{grading} on $A$ by a (semi)group $S$
if $A^{(s)}A^{(t)}\subseteq A^{(st)}$ for all $s,t\in S$.
Then we say that $S$ is the \textit{grading (semi)group} of $\Gamma$ and the algebra $A$ is \textit{graded} by $S$.

Let \begin{equation}\label{EqTwoSemiGroupGradings}\Gamma_1 \colon A=\bigoplus_{s \in S} A^{(s)},\qquad \Gamma_2
\colon B=\bigoplus_{t \in T} B^{(t)}\end{equation} be two gradings where $S$
and $T$ are (semi)groups and $A$ and $B$ are algebras.

The most restrictive case is when we require that both grading (semi)groups
coincide:

\begin{definition}[{e.g. \cite[Definition~1.15]{ElduqueKochetov}}]
\label{DefGradedIsomorphism}
The gradings~(\ref{EqTwoSemiGroupGradings}) are \textit{isomorphic} if $S=T$ and there exists an isomorphism $\varphi \colon A \mathrel{\widetilde\to} B$
of algebras such that $\varphi(A^{(s)})=B^{(s)}$
for all $s\in S$.
In this case we say that $A$ and $B$ are \textit{graded isomorphic}.
\end{definition}

In some cases, such as in~\cite{ginosargradings}, less restrictive requirements are more
suitable. 
\begin{definition}[{\cite[Definition~2.3]{ginosargradings}}]\label{DefGinosarEquivalent}
The gradings~(\ref{EqTwoSemiGroupGradings})  are \textit{equivalent} if there exists an isomorphism
$\varphi \colon A \mathrel{\widetilde\to} B$
of algebras and an isomorphism $\psi \colon S \mathrel{\widetilde\to} T$
of (semi)groups such that $\varphi(A^{(s)})=B^{\bigl(\psi(s)\bigr)}$
for all $s\in S$.
\end{definition}
\begin{remark}
The notion of graded equivalence was considered by Yu.\,A.~Bahturin, S.\,K.~Seghal, and M.\,V.~Zaicev in~\cite[Remark after Definition 3]{BahturinZaicevSeghalGroupGrAssoc}.
In the paper of V.~Mazorchuk and K.~Zhao~\cite{Mazorchuk} it appears under the name of graded isomorphism.
A.~Elduque and M.\,V.~Kochetov refer to this notion as a weak isomorphism
of gradings~\cite[Section~3.1]{ElduqueKochetov}.
More on differences in the terminology in
graded algebras can be found in \cite[\S 2.7]{ginosargradings}.
\end{remark}

 If one studies the graded structure of a graded algebra or its graded polynomial identities~\cite{AljaGia, AljaGiaLa,
BahtZaiGradedExp, GiaLa, ASGordienko9}, then it is not really important by elements of which (semi)group the graded components are indexed.
A replacement of the grading (semi)group leaves both graded subspaces and graded ideals graded.
In the case of graded polynomial identities reindexing
the graded components leads only to renaming the variables.
(It is important to notice however that graded-simple algebras
graded by semigroups which are not groups can have a structure quite different from group graded graded-simple algebras~\cite{GordienkoJanssensJespers}.)
 Here we come naturally to the notion of weak equivalence of gradings.

\begin{definition}\label{def:weakly}
The gradings~(\ref{EqTwoSemiGroupGradings}) are \textit{weakly equivalent}, if there
exists an isomorphism $\varphi \colon A \mathrel{\widetilde\to}B$
of algebras such that for every $s\in S$ with $A^{(s)}\ne 0$ there
exists $t\in T$ such that $\varphi\left(A^{(s)}\right)=B^{(t)}$.
\end{definition}
 \begin{remark}
 This notion appears in~\cite[Definition~1.14]{ElduqueKochetov}
 under the name of equivalence.
 We have decided to add here the adjective ``weak'' in order to avoid confusion
 with Definition~\ref{DefGinosarEquivalent}.
 \end{remark}

For a grading $\Gamma \colon A=\bigoplus_{s \in S} A^{(s)}$, we denote by $\supp \Gamma := \lbrace s\in S \mid A^{(s)}\ne 0\rbrace$ its support.
\begin{remark}Each weak equivalence $\varphi$ between gradings $\Gamma_1$ and $\Gamma_2$
 induces a bijection $\psi \colon \supp \Gamma_1 \mathrel{\widetilde\to} \supp \Gamma_2$ defined by $\varphi\left( A^{(s)}\right) = B^{\left(\psi(s)\right)}$ for $s\in \supp \Gamma_1$.
\end{remark}

Obviously, if gradings are isomorphic, then they are equivalent and if they are equivalent then they are also weakly equivalent.
It is important to notice that none of the converse is true.
However, if gradings~(\ref{EqTwoSemiGroupGradings}) are weakly equivalent
and $\varphi \colon A \mathrel{\widetilde\to}B$ is the corresponding isomorphism of algebras,
then $\Gamma_3 \colon A=\bigoplus_{t \in T} \varphi^{-1}\left( B^{(t)}\right)$ is a grading on $A$ isomorphic to $\Gamma_2$ and the grading $\Gamma_3$ is obtained from $\Gamma_1$ just by reindexing the homogeneous components.
Therefore, when gradings~(\ref{EqTwoSemiGroupGradings}) are weakly
equivalent, we say that $\Gamma_1$ \textit{can be regraded} by $T$.
If $A=B$ and $\varphi$ in Definition~\ref{def:weakly} is the identity map, we say that $\Gamma_1$ and $\Gamma_2$ are realizations
of the same grading on $A$ as, respectively, an $S$- and a $T$-grading.

Note that $\Gamma$ is obviously equivalent to $\Gamma_0 \colon A=\bigoplus_{s \in S_1} A^{(s)}$
where $S_1$ is a subsemigroup of $S$ generated by $\supp \Gamma$.
(If $S$ is a group, we can consider instead the subgroup generated by $\supp \Gamma$.)

As we have already mentioned above, for many applications it is not important
which particular grading among weakly equivalent ones
we consider. Thus, if it is possible, one can try to regrade a semigroup grading by a group or even a finite group. The situation, when the latter
is possible, is very convenient since the algebra graded by a finite group $G$
is an $FG$-comodule algebra and, in turn, an $(FG)^*$-module algebra
where $FG$ is the group algebra of $G$, which is a Hopf algebra, and $(FG)^*$ is its dual.
In this case one can use the techniques of Hopf algebra actions instead of working with a grading directly (see e.g.~\cite{ASGordienko3}).
Therefore, the following question arises naturally:

\begin{question}
Is it possible to regrade any grading of a finite dimensional
algebra by a finite group ?
\end{question}

It is fairly easy to show (see Proposition~\ref{prop:abelian} below) that each grading of a finite dimensional algebra by an abelian group $G$ is weakly equivalent to a grading by a finite group.
In fact, the class of groups $G$ with this property is much broader
and includes at least all locally residually finite groups (see~\cite[Proposition 1.2]{DasNasDelRioVanOyst} and Theorem~\ref{TheoremFinGradingsToGroupsTransition} below). Recall that a group is \textit{residually finite} if the intersection of its normal subgroups of finite index is trivial. A group  $G$ is \textit{locally residually finite} if every finitely generated subgroup of $G$ is residually finite.

In 1996 M.\,V.~Clase, E.~Jespers, and \'A.~Del R\'\i o~\cite[Example~2]{ClaseJespersDelRio} (see also~\cite[Example 1.5]{DasNasDelRioVanOyst}) gave an example of a group graded ring with finite support that cannot be regraded by a finite (semi)group. Despite the fact that they constructed a ring, not an algebra, it is obvious how to make an analogous example of an algebra over a field. However all these examples would have non-trivial nilpotent ideals. Until now it was unclear whether a finite dimensional semi-simple algebra could have a grading
which cannot be regraded by a finite group.

In Theorem~\ref{TheoremFiniteRegradingImpossible} below we show that
there exist even elementary gradings (see Definition~\ref{def:elementary}) on the full matrix algebras $M_n(F)$ (where $F$ is a field) that are not weakly equivalent to gradings by finite groups.
This suggests the following problem:

\begin{problem}\label{ProblemMinNElAGradNonFin} Determine
the set $\Omega$ of  
the numbers $n\in\mathbb N$ such that any elementary grading on $M_n(F)$ can be regraded by a finite group.
\end{problem}

In Section~\ref{SectionRegradingMatrixFinite} we prove the following theorem:

\begin{theorem}\label{TheoremRegradeElementaryOmega}
The set $\Omega$ defined in Problem~\ref{ProblemMinNElAGradNonFin}
is of the form $\lbrace n \in \mathbb N \mid 1 \leqslant n \leqslant n_0\rbrace$ for some $3\leqslant n_0 \leqslant 348$.
In particular, for every $n \geqslant 349$ there exists an elementary grading on $M_n(F)$ that cannot be regraded by any finite group.
\end{theorem}

In Section~\ref{SectionWeakEquivalenceGrSimpleAlg} we provide a criterion for two group gradings on graded-simple algebras to be weakly equivalent
and give an example of two twisted group algebras of the same group
which are isomorphic as algebras, but whose standard gradings
are not weakly equivalent.

In Section~\ref{SectionGroupTheoreticalApproach} we recall the definition of the universal group of a grading  introduced
in 1989 by J.~Patera and H.~Zassenhaus~\cite{PZ89} and prove that any finitely presented group $G$ can be a universal group of an elementary grading $\Gamma$ on a full matrix algebra and moreover any finite subset of $G$ can be included in $\supp \Gamma$ (Theorem~\ref{TheoremGivenFinPresGroupExistence}). This result is used in the proof of Theorem~\ref{TheoremRegradeElementaryOmega}. We conclude the section showing that the question whether a given grading is regradable by a finite group and, in particular, Problem~\ref{ProblemMinNElAGradNonFin} can be reformulated in a purely group theoretical way (see Theorem~\ref{TheoremFinGradingsToGroupsTransition} and Problem~\ref{ProblemMinFnHRFRx1x2x3}).

In Section~\ref{SectionRegradingMatrixFinite} we prove Corollary~\ref{CorollaryFiniteRegradingImpossible349}, where we show that for every $n \geqslant 349$ there exists an elementary grading on $M_n(F)$ that cannot be regraded by any finite group, and Theorem~\ref{TheoremFiniteRegradingSmallMatrix}, where we show that 
if $n\leqslant 3$, then any elementary grading on $M_n(F)$ is weakly equivalent to an elementary grading by a finite group. Together this proves Theorem~\ref{TheoremRegradeElementaryOmega}.

\section{Preliminaries}
\label{SectionGrSimpleAlg}

In this section we recall some basic facts about graded algebras.

Let $S$ be a semigroup and let $\Gamma \colon A=\bigoplus_{s\in S} A^{(s)}$ be a grading. The subspaces $A^{(s)}$ are called \textit{homogeneous components} of $\Gamma$
and nonzero elements of $A^{(s)}$ are called \textit{homogeneous} with respect to~$\Gamma$. A subspace $V$ of  $A$ is \textit{graded} if $V=\bigoplus_{t\in S} (V \cap A^{(s)})$.
If $A$ does not contain graded two-sided ideals, i.e. two-sided ideals that are graded subspaces,
then $A$ is called \textit{graded-simple}.

\subsection{Lower dimensional group cohomology and graded division algebras}
\label{SubsectionCohomologyGradedDivisionAlgebras}

By the Bahturin~--- Seghal~--- Zaicev Theorem (Theorem~\ref{TheoremBahturinZaicevSeghal} below)
twisted group algebras, which are graded division algebras, play a crucial role in the classification of graded-simple algebras.

Let $A=\bigoplus_{g\in G} A^{(g)}$ be a $G$-graded algebra for some group $G$.
If for every $g\in G$ all nonzero elements of $A^{(g)}$ are invertible, then $A$ is called a \textit{graded division algebra}.

Finite dimensional $G$-graded division algebras over an algebraically closed field are described by the elements of the second cohomology groups of finite subgroups $H \subseteq G$ with coefficients in the multiplicative group of the base field.

Let $G$ be a group and let $F$ be a field. Denote by $F^{\times}$ the multiplicative group of $F$. Throughout the article we consider only trivial group actions on $F^{\times}$. In this case the \textit{first cohomology 
group} $H^1(G, F^{\times})$ is isomorphic to the group $Z^1(G, F^{\times})$
of \textit{$1$-cocycles} which in turn coincides with the group $\Hom(G,F^\times)$ of group homomorphisms $G \to F^\times$ with the pointwise multiplication.

Recall that a function $\sigma\colon G \times G \to F^{\times}$ is a \textit{$2$-cocycle} if $\sigma(u,v)\sigma(uv,w)=\sigma(u,vw)\sigma(v,w)$ for all $u,v,w\in G$. The set $Z^2(G, F^{\times})$ of $2$-cocycles is an abelian group with respect to the pointwise multiplication. 
The subgroup $B^2(G, F^{\times}) \subseteq Z^2(G, F^{\times})$
of \textit{$2$-coboundaries} consists of all $2$-cocycles $\sigma$
for which there exists a map $\tau \colon G \to F^\times$
such that we have $\sigma(g,h)=\tau(g)\tau(h)\tau(gh)^{-1}$
for all $g,h\in G$.
The factor group $H^2(G,F^\times) := Z^2(G, F^{\times})/B^2(G, F^{\times})$ 
is called the \textit{second cohomology group of $G$ with coefficients in $F^\times$}.
Denote by $[\sigma]$ the cohomology class of $\sigma \in Z^2(G, F^{\times})$
in $H^2(G,F^\times)$.

Let $\sigma \in Z^2(G, F^{\times})$.
The \textit{twisted group algebra} $F^\sigma G$ is the associative algebra
with the formal basis $(u_g)_{g\in G}$ and the multiplication $u_g u_h = \sigma(g,h)u_{gh}$
for all $g,h \in G$.
For \textit{trivial} $\sigma$, i.e. when $\sigma(g,h)=1$ for all $g,h\in G$, the twisted group algebra $F^\sigma G$ is the ordinary \textit{group algebra} $FG$. Each twisted group algebra $F^\sigma G$ has the \textit{standard grading} $F^{\sigma} G=\bigoplus_{g\in G}F^{\sigma} G^{(g)}$ where $F^{\sigma} G^{(g)}=Fu_g$. Two twisted group algebras $F^{\sigma_1} G$ and $F^{\sigma_2} G$
are graded isomorphic if and only if $[\sigma_1]=[\sigma_2]$. (See e.g.~\cite[Theorem 2.13]{ElduqueKochetov}.)

Note that twisted group algebras are graded division algebras. In fact,
the component $B^{(1_G)}$ of an arbitrary graded division algebra $B$, that corresponds to the neutral element $1_G$ of the grading group $G$, is an ordinary division algebra. Thus, if the base field $F$ is algebraically
closed and $\dim B$ is finite, we have $B^{(1_G)} = F1_B$ and $B \cong F^\sigma H$ for some finite subgroup $H \subseteq G$ and a $2$-cocycle $\sigma \in Z^2(H, F^{\times})$. (See~\cite[Theorem 2.13]{ElduqueKochetov} for the details.)

Suppose there exists
a homomorphism $\varphi \colon F^\sigma G
\to F$ of unital algebras. Then $$\varphi(u_g)\varphi(u_h)=\varphi(u_g u_h)
=\sigma(g,h)\varphi(u_{gh})$$ and
$\sigma(g,h)=\varphi(u_g)\varphi(u_h)\varphi(u_{gh})^{-1}$ for all $g,h\in G$, i.e. $\sigma$ is cohomologous to the trivial $2$-cocycle.
Consequently, if $[\sigma]$ is non-trivial, then $F^\sigma G$ does not have
one dimensional unital modules.

Recall that if $G$ is finite and $\ch F \nmid |G|$, then $F^\sigma G$ is semisimple. (The proof is completely analogous to the case of an ordinary group algebra, see e.g.~\cite[Theorem~1.4.1]{Herstein}.) Therefore, if $G$ is finite, $\ch F \nmid |G|$, and the field $F$ is algebraically closed,
the Artin--Wedderburn Theorem implies that $F^\sigma G$ is
isomorphic to the direct sum of full matrix algebras $M_k(F)$.
In the case $[\sigma]$ is non-trivial, the observation in the previous paragraph shows
 that $k\geqslant 2$ for all $M_k(F)$.
Unlike ordinary group algebras $FG$ of non-trivial groups, twisted group algebras $F^\sigma G$
can be simple. (See e.g.~\cite[Theorem 2.15]{ElduqueKochetov}.)

Let $G$ be an abelian group and $\sigma\in Z^2(G, F^{\times})$.
Then in $F^\sigma G$ we have $u_g u_h = \beta(g,h) u_h u_g$
where $\beta(g,h) := \sigma(g,h)\sigma(h,g)^{-1}$, $g,h \in G$,
is the alternating bicharacter corresponding
to $\sigma$. Recall that a function $\beta \colon G\times G \to F^\times$
is an \textit{alternating bicharacter} if it is multiplicative in each variable
and $\beta(g,g)=1_F$ for all $g\in G$.
It is easy to see that $\beta$ depends only on the cohomology class $[\sigma]\in H^2(G, F^{\times})$
and not on the particular $2$-cocycle $\sigma$.

If $G$ is a finitely generated abelian group, then $G \cong (\mathbb Z/n_1 \mathbb Z) \times (\mathbb Z/n_2 \mathbb Z)\times \dots \times (\mathbb Z/n_m \mathbb Z)$ for some non-negative integers $m, n_i$.
In this case, in order to define an alternating bicharacter $\beta \colon G\times G \to F^\times$, it is necessary and sufficient to define the values $\beta(g_i,g_j)$
where $$\beta(g_i,g_j)^{n_i}=\beta(g_i,g_j)^{n_j}=\beta(g_i,g_j)\beta(g_j,g_i)=\beta(g_i,g_i)=1$$
for all $1\leqslant i,j \leqslant m$
and $g_i$ are generators of the cyclic components of $G$.

Given an alternating bicharacter $\beta \colon G\times G \to F^\times$, it is easy to define an algebra which is graded isomorphic to a twisted group algebra $F^\sigma G$ with $[\sigma]$
corresponding to $\beta$:
$$u_{g_1^{k_1}\cdots g_m^{k_m}} u_{g_1^{\ell_1}\cdots g_m^{\ell_m}}
= \left(\prod_{1\leqslant i < j \leqslant m} \beta(g_j,g_i)^{k_j \ell_i}\right)u_{g_1^{k_1+\ell_1}\cdots g_m^{k_m+\ell_m}}.$$

Similar arguments show that if $\sigma_1,\sigma_2\in Z^2(G, F^{\times})$
have equal alternating bicharacters, then $F^{\sigma_1}G$
and $F^{\sigma_2}G$ are graded isomorphic, and $[\sigma_1]=[\sigma_2]$.

\subsection{Elementary gradings and classification of graded-simple algebras}
\label{SubsectionElementaryGradedSimpleAlgebras}

\begin{definition}\label{def:elementary}
Let $F$ be a field, $G$ be a group, let $n\in\mathbb N$, and let $(g_1, \dots, g_n)$ be an $n$-tuple of elements
of $G$. Define a grading on $M_n(F)$ by making each matrix unit $e_{ij}$ a $g_i g_j^{-1}$-homogeneous element. This grading is called the \textit{elementary $G$-grading} defined by $(g_1, \dots, g_n)$.
\end{definition}
\begin{remark}\label{RemarkElementary}
Note that such a grading is uniquely determined by defining the $G$-degrees of $e_{i,i+1}$, $1\leqslant i \leqslant n-1$. If $G$ is an arbitrary group and $(h_1, \dots, h_{n-1})$ is an arbitrary $(n-1)$-tuple of elements of $G$, then the elementary grading with $e_{i,i+1} \in M_n(F)^{(h_i)}$ can be defined by $(g_1, \dots, g_n)$ where $g_i = \prod_{j=i}^{n-1} h_j$.
\end{remark}

Let $n\in\mathbb N$, let $G$ be a group, let $\gamma=(g_1, \dots, g_n)$ where $g_i \in G$, let $H \subseteq G$ be a finite subgroup, and let $\sigma \in Z^2(H, F^{\times})$.
Denote by $M(\gamma, \sigma)$ the algebra $M_n(F)\otimes_F F^\sigma H$
endowed with the grading where $e_{ij}\otimes u_h$ belongs to the
$g_i h g_j^{-1}$-component.

Recall
the following classification result:

\begin{theorem}[Bahturin~--- Seghal~--- Zaicev, see e.g.~{\cite[Theorem~3]{BahturinZaicevSeghalSimpleGraded} or~\cite[Corollary 2.22]{ElduqueKochetov}}]
\label{TheoremBahturinZaicevSeghal} Let $A$ be a finite dimensional graded-simple $G$-graded algebra
over an algebraically closed field $F$ where $G$ is a group.
Then $A$ is graded isomorphic to $M(\gamma, \sigma)$ for some
$n\in\mathbb N$, $\gamma=(g_1, \dots, g_n)$ where $g_i \in G$, a finite subgroup $H \subseteq G$,
and a $2$-cocycle $\sigma \in Z^2(H, F^{\times})$.
\end{theorem}

A criterion for two such gradings to be isomorphic can be found, e.g., in~\cite[Lemma 1.3, Proposition 3.1]{AljaHaile} or~\cite[Corollary 2.22]{ElduqueKochetov}.
Necessary and sufficient conditions for two graded-simple algebras to be graded equivalent were proven in~\cite[Theorem 2.20]{ginosargradings}. In the next section we study weak equivalences
of gradings on two graded-simple algebras.

\section{Weak equivalences of graded-simple algebras}
\label{SectionWeakEquivalenceGrSimpleAlg}

In this section we prove a criterion for a weak equivalence of graded-simple algebras
inspired by~\cite[Proposition~2.33]{ElduqueKochetov},
we present families of gradings for which the notions of equivalence and weak equivalence coincide,
and give an example of two twisted group algebras of the same abelian group that are  isomorphic as ordinary algebras, but not
graded weakly equivalent.

Let $A=\bigoplus_{g\in G} A^{(g)}$ be an algebra graded by a group $G$.
A vector space $W=\bigoplus_{g\in G} W^{(g)}$ is a \textit{graded left $A$-module} if
$W$ is a left $A$-module and for each $g,h\in G$
we have $A^{(g)}W^{(h)} \subseteq W^{(gh)}$. Graded right modules are defined analogously.

Fix some $n\in\mathbb N$, a group $G$, an $n$-tuple $\gamma=(g_1, \dots, g_n)$ where $g_i \in G$, a finite subgroup $H \subseteq G$, and $\sigma \in Z^2(H, F^{\times})$.

Consider the $G$-graded vector space $V$ with the basis $v_{i,h}$, $1\leqslant i \leqslant n$, $h\in H$, such that $v_{i,h}$ is a homogeneous element of degree $g_i h$.
Then $V$ is a graded left $M(\gamma, \sigma)$-module
and a graded right
$F^{\sigma}H$-module
 with $(e_{jk}\otimes u_g) v_{i,h} := \delta_{ki} \sigma(g,h) v_{j,gh}$
and $v_{i,h} u_g := \sigma(h,g) v_{i,hg}$
for all $1\leqslant i,j, k \leqslant n$ and $g,h \in H$
where $\delta_{ij}=\left\lbrace\begin{array}{rrr} 0 & \text{if} & i\ne j,\\ 1 & \text{if} & i=j \end{array}\right.$ and $(u_g)_{g\in G}$ is the standard basis in $F^{\sigma}H$.
Note that $(av)u = a (v u)$ for all $a\in M(\gamma, \sigma)$, $v\in V$, and $u\in F^{\sigma}H$.

\begin{lemma}\label{LemmaVIrrLeftModule}
$V$ is an irreducible graded left $M(\gamma, \sigma)$-module.
\end{lemma}
\begin{proof}
Suppose $W$ is a non-trivial graded $M(\gamma, \sigma)$-submodule of $V$.
Take non-zero $v=\bigoplus_{i,h} \alpha_{i,h} v_{i,h} \in W$. If $\alpha_{i,h_0}\ne 0$
for some $1\leqslant i \leqslant n$ and $h_0\in H_i$, then
  $e_{ii}\otimes u_{1_H} v= \bigoplus_{h \in H} \alpha_{i,h} v_{i,h} \in W$
  is again non-zero. Since $W$ is a graded subspace and for fixed $i$ and different $h$
  the elements $v_{i,h}$ belong to different graded components, we get
  $\alpha_{i, h_0} v_{i, h_0} \in W$ and $v_{i, h_0} \in W$. Hence
  $v_{j,h}= \frac{1}{\sigma(hh_0^{-1},h_0)}(e_{ji}\otimes u_{hh_{0}^{-1}})v_{i,h_0} \in W$ for all $1\leqslant j \leqslant n$, $h\in H$.
  Therefore $W=V$.
\end{proof}

  For a graded vector space $W=\bigoplus_{g\in G} W^{(g)}$ and an element $h\in G$
  denote by $W^{[h]}$ the same vector space endowed with the grading
  $W=\bigoplus_{g\in G} \tilde W^{(g)}$ where $\tilde W^{(g)} := W^{(gh^{-1})}$.

  Note that $M(\gamma, \sigma)$ is the direct sum of graded left ideals
  $I_j=\bigoplus_{\substack{h\in H, \\ 1\leqslant i \leqslant n}} F (e_{ij}\otimes u_h)$,
  $1\leqslant j \leqslant n$, and each $I_j$ is isomorphic to $V^{[g_j^{-1}]}$
  as a graded left $M(\gamma, \sigma)$-module via $e_{ij}\otimes h \mapsto v_{i,h}$,
  $1\leqslant i \leqslant n$, $h\in H$.
  
  Recall that if $\alpha \colon H_1 \to H_2$
  is a homomorphism of groups and $\sigma \in Z^2(H_2, F^{\times})$,
  then the function $\sigma(\alpha(g), \alpha(h))$ of arguments $g,h\in H_1$
  is a $2$-cocycle on $H_1$. We denote the cohomology class of this $2$-cocycle by 
  $[\sigma(\alpha(\cdot),\alpha(\cdot))]$.

  Let $G_1,G_2$ be groups and let $n_1,n_2\in\mathbb N$. Fix tuples $\gamma_i=(g_{i1}, \dots, g_{in_i})$
where $g_{ij} \in G_i$, finite subgroups $H_i \subseteq G_i$, and $2$-cocycles
$\sigma_i \in Z^2(H_i, F^{\times})$, $i=1,2$.
We say that the gradings on $M(\gamma_1, \sigma_1)$ and $M(\gamma_2, \sigma_2)$
satisfy Condition (*) if $n_1=n_2$, there exist a group isomorphism $\alpha \colon H_1 \mathrel{\widetilde{\to}} H_2$,
a permutation $\pi \in S_{n_1}$,
and elements $t_i \in H_2$, $1\leqslant i \leqslant n_1$,
such that $[\sigma_1] = [\sigma_2(\alpha(\cdot),\alpha(\cdot))]$
and the following condition holds:
for every $1\leqslant i,j,k,\ell \leqslant n_1$ and $h_1,h_2\in H_1$
we have $$g_{2,\pi(i)}t_i\alpha(h_1)t_j^{-1}g_{2,\pi(j)}^{-1} = g_{2,\pi(k)} t_k \alpha(h_2) t_\ell^{-1} g_{2,\pi(\ell)}^{-1}$$
if and only if
$$g_{1i}h_1 g_{1j}^{-1} = g_{1k} h_2 g_{1\ell}^{-1}.$$

 \begin{theorem}\label{TheoremWeakEquivGradedSimple}
 Let $G_1,G_2$ be groups and let $n_1,n_2\in\mathbb N$. Fix tuples $\gamma_i=(g_{i1}, \dots, g_{in_i})$
where $g_{ij} \in G_i$, finite subgroups $H_i \subseteq G_i$, and $2$-cocycles
$\sigma_i \in Z^2(H_i, F^{\times})$, $i=1,2$.
The gradings on $M(\gamma_1, \sigma_1)$ and $M(\gamma_2, \sigma_2)$
are weakly equivalent if and only if they satisfy Condition (*).
If $M(\gamma_1, \sigma_1)$ and $M(\gamma_2, \sigma_2)$ satisfy Condition (*), the algebra isomorphism
 $\varphi \colon M(\gamma_1, \sigma_1) \mathrel{\widetilde{\to}} M(\gamma_2, \sigma_2)$
 implementing this weak equivalence
can be defined e.g. by $\varphi(e_{ij}\otimes x_h) = e_{\pi(i),\pi(j)}\otimes y_{t_i\alpha(h)t_j^{-1}}$
where $(x_h)_{h\in H_1}$ is the formal basis in $F^{\sigma_1} H_1$ and $(y_t)_{t\in H_2}$ is the formal basis in $F^{\sigma_2} H_2$.
\end{theorem}
\begin{proof} Suppose the gradings on $M(\gamma_1, \sigma_1)$ and $M(\gamma_2, \sigma_2)$
are weakly equivalent.
  Denote by $\varphi$ an isomorphism of algebras $M(\gamma_1, \sigma_1)\mathrel{\widetilde{\to}} M(\gamma_2, \sigma_2)$ that corresponds to the weak equivalence.
Construct the graded $M(\gamma_i, \sigma_i)$-modules $V_i$, $i=1,2$, as above.
These modules are irreducible by Lemma~\ref{LemmaVIrrLeftModule}.

Now we use the isomorphism $\varphi$ to define
on $V_2$ the structure of a left (non-graded) $M(\gamma_1, \sigma_1)$-module
via $a \cdot v= \varphi(a)v$ for $a\in M(\gamma_1, \sigma_1)$  and $v\in V_2$.
 Define for $M(\gamma_1, \sigma_1)$ the minimal $G_1$-graded left ideals $I_j$ as above. Then for each $j$ the space $\varphi(I_j)$ is a minimal $G_2$-graded 
left ideal of $M(\gamma_2, \sigma_2)$. For a homogeneous element $v\in V_2$
the spaces $I_j \cdot v$ are $G_2$-graded $M(\gamma_2, \sigma_2)$-submodules
of $V_2$.
Since by Lemma~\ref{LemmaVIrrLeftModule} the module
$V_2$ does not contain any non-zero proper $G_2$-graded $M(\gamma_2, \sigma_2)$-submodules, for every $j$ and every homogeneous $v\in V_2$
we have either $I_j \cdot v = 0$ or $I_j \cdot v = V_2$. Since $M(\gamma_1, \sigma_1) \cdot V_2 \ne 0$, we obtain that $V_2=I_j\cdot v$ for some homogeneous $v\in V_2$ and some $j$. 
Hence $V_2$ is isomorphic to $I_j$ as a (non-graded) left $M(\gamma_1, \sigma_1)$-module. Moreover, this isomorphism maps each nonzero $G_1$-graded component onto some
$G_2$-graded component. In fact, since $I_j \cong V_1^{[g_j^{-1}]}$,
there exists a linear isomorphism $\psi \colon V_1 \mathrel{\widetilde{\to}} V_2$
such that $\varphi(a)\psi(v)=\psi(av)$ for all $a\in M(\gamma_1, \sigma_1)$
and $v\in V_1$ and for each $g\in G_1$ with $V_1^{(g)}\ne 0$ there exists $\rho(g)\in G_2$
such that $\psi\left(V_1^{(g)}\right)=V_2^{(\rho(g))}$.

It is not difficult to check that $F^{\sigma_i} H_i$ 
is isomorphic to $\End_{M(\gamma_i, \sigma_i)} V_i$ as an algebra through its action on $V_i$ from the right. Hence there exists an algebra isomorphism $\tau \colon F^{\sigma_1} H_1 \mathrel{\widetilde{\to}} F^{\sigma_2} H_2$ such that $\psi(v)\tau(u)=\psi(vu)$ where $v\in V_1$ and $u\in F^{\sigma_1} H_1$.
 If $V_2^{(g)} \ne 0$
and $h\in H_1$, then \begin{equation}\label{EqWeakGrSmpl1}V_2^{(g)}\tau(x_h)=\psi\left(V_1^{\left(\rho^{-1}(g)\right)}\right)\tau(x_h)=
\psi\left(V_1^{\left(\rho^{-1}(g)\right)}x_h\right)=\psi\left(V_1^{\left(\rho^{-1}(g)h\right)}\right)=
V_2^{\left(\rho\left(\rho^{-1}(g)h\right)\right)}.\end{equation}
However, $\tau(x_h)=\sum_{t\in H_2} \alpha_t y_t$ for some $\alpha_t\in F$.
Since the sum $\bigoplus_{t\in H_2} V_2^{(g)}y_t = \bigoplus_{t\in H_2} V_2^{(gt)}$
is direct, \begin{equation}\label{EqWeakGrSmpl2}\tau(x_h)=\lambda(h) y_{\alpha(h)}\end{equation} for some $\lambda(h) \in F^{\times}$
and $\alpha(h) \in H_2$. Since $\tau$ is an algebra isomorphism, $\alpha$ is a group isomorphism.
Now \begin{equation*}\begin{split}\sigma_1(h_1, h_2)\lambda(h_1h_2)y_{\alpha(h_1 h_2)}
=\sigma_1(h_1, h_2)\tau(x_{h_1 h_2})=\\=\tau(x_{h_1})\tau(x_{h_2})=\sigma_2(\alpha(h_1),\alpha(h_2))
\lambda(h_1)\lambda(h_2)y_{\alpha(h_1) \alpha(h_2)},\end{split}\end{equation*}
implies $[\sigma_1] = [\sigma_2(\alpha(\cdot),\alpha(\cdot))]$.

Note that $\dim V_1 = n_1 |H_1|=\dim V_2 = n_1 |H_2|$. Now $H_1 \cong H_2$
implies $n_1 = n_2$.

Equalities~(\ref{EqWeakGrSmpl1}) and~(\ref{EqWeakGrSmpl2}) imply
$\rho\left(gh\right) = \rho(g)\alpha(h)$ for all $g\in G_1$ with $V_1^{(g)}\ne 0$
and all $h\in H_1$.
Hence $\rho$ is a bijection between $\bigcup_{i=1}^{n_1} g_{1i}H_1$ and $\bigcup_{i=1}^{n_1} g_{2i}H_2$
which maps left cosets of $H_1$ onto left cosets of $H_2$.
Since for each fixed $1\leqslant i \leqslant n_1$ and each $h\in H_1$
the number $$\dim\left(V_1^{(g_{1i} h)}\right)=\dim\left(V_2^{\bigl(\rho(g_{1i}) \alpha(h)\bigr)}\right)$$ equals the number of different $1\leqslant j \leqslant n_1$ such that $g_{1j} \in g_{1i}H_1$
and the number of different $1\leqslant j \leqslant n_1$ such that $g_{2j} \in \rho(g_{2i})H_2$,
there exists a permutation $\pi \in S_{n_1}$ and elements $t_i \in H_2$ such that $\rho(g_{1i})=g_{2,\pi(i)} t_i$, $1\leqslant i \leqslant n_1$.

Denote by $v_{i,h}$ the elements of the standard basis in $V_1$ defined before the theorem.
Now $\varphi(e_{ij}\otimes x_h) \psi(v_{j,1_{H_1}})=\sigma(h,1_{H_1})\psi(v_{i,h})$.
Note that $\deg \psi(v_{j,1_{H_1}}) = \rho(g_{1j})= g_{2, \pi(j)} t_j$ and
$\deg \psi(v_{i,h}) = \rho(g_{1i}h)= g_{2, \pi(i)}t_i\alpha(h)$.
Hence $\deg \varphi(e_{ij}\otimes x_h) = g_{2, \pi(i)}t_i\alpha(h) t_j^{-1} g_{2, \pi(j)}^{-1}$
for all $1\leqslant i,j\leqslant n_1$.
Since $\varphi$ is a weak equivalence of gradings and $\deg (e_{ij}\otimes x_h) = g_{1i}hg_{1j}^{-1}$,
we get the first part of the theorem.

The converse is trivial.
\end{proof}

The proposition below is verified directly.
\begin{proposition}
An elementary grading on a full matrix algebra can be weakly equivalent only to a
grading isomorphic to an elementary grading on a full matrix algebra.
\end{proposition}

Let $G$ be a group. We say that a $G$-grading of an algebra $A$ is
\textit{connected} if the support of this grading generates the
group $G$. Recall that a grading $A=\bigoplus_{g \in G} A^{(g)}$
is called \textit{strong} if $A^{(g_1)}A^{(g_2)}=A^{(g_1g_2)}$ for
any $g_1,g_2\in G$ and a grading is called \textit{nondegenerate}
if the product of a finite number of non-zero homogeneous
components is again non-zero.

It is easy to see that if $A\ne 0$ is strongly graded, then $A^{(g)}\ne 0$ 
for at least one $g\in G$ and $A^{(gh^{-1})}A^{(h)}=A^{(g)}\ne 0$ implies $A^{(h)}\ne 0$ for all $h\in G$. In particular, a strong grading is connected.

Here we introduce the notion of
a strongly connected grading which is weaker than the notion
of a connected nondegenerate grading and a strong grading (in the case 
of a nonzero algebra).

\begin{definition}
A connected grading $\Gamma\colon A=\bigoplus_{g \in G} A^{(g)}$ is \textit{strongly connected} if $A^{(g)}A^{(h)}\ne 0$
for all $g,h\in \supp\Gamma$.
\end{definition}
\begin{lemma}
Weakly equivalent strongly connected gradings of finite
dimensional algebras are equivalent.
\end{lemma}
\begin{proof}
Let $\Gamma_1 \colon A=\bigoplus_{g \in G} A^{(g)}$ be a strongly
connected grading. We claim that $\supp\Gamma _1$ coincides with $G$ itself. 
Take arbitrary $g=\prod _{i=1}^r g_i^{k_i}$
where $g_1,g_2,\dots ,
g_r\in\supp\Gamma _1$ and $k_1,k_2,\dots ,k_r\in \mathbb{N}$. Now, using the strongly
connectedness condition, we get by induction on $r$ that $A^{(g)}\neq 0$ and therefore
$\supp\Gamma_1$ is closed under multiplication. 
 Since $A$ is finite dimensional,
$\supp\Gamma_1$ is finite subset of a group, which is closed under multiplication.
Hence  $\supp\Gamma_1$ is a group and $\supp\Gamma_1 = G$ since $\Gamma_1$ is connected.

Let $\Gamma_2 \colon B=\bigoplus_{h \in H} B^{(h)}$ be a strongly
connected grading which is weakly equivalent to $\Gamma _1$ with
the associated isomorphism $\varphi \colon A\rightarrow B$. By the
arguments above, $\supp\Gamma _2=H$. Therefore, the natural
bijection between the supports of the gradings is a map $\psi \colon G\rightarrow H$ between
the grading groups. We claim that $\psi $ is a
group isomorphism. Indeed, let $g,g'\in G$. Then $A^{(g)}$ and $A^{(g')}$ are
both non-zero. Suppose $\varphi\left(A^{(g)}\right)=B^{(h)}$, $\varphi\left(A^{(g')}\right)=B^{(h')}$. Then $\psi(g)=h$ and $\psi(g')=h'$. Since $\varphi$ is an isomorphism
of algebras, we have
$$\varphi \left(A^{(g)}A^{(g')}\right)=\varphi\left(A^{(g)}\right)\varphi\left(A^{(g')}\right)=B^{(h)}B^{(h')}\subseteq B^{(hh')}.$$
On the other hand,
$$\varphi\left(A^{(g)}A^{(g')}\right)\subseteq \varphi\left(A^{(gg')}\right).$$
Consequently, since by the strongly connectedness condition
$A^{(g)}A^{(g')}\neq 0$, we get $\varphi\left(A^{(gg')}\right)=B^{(hh')}$ and
therefore $\psi(gg')=hh'$. Hence $\psi$ is indeed a group isomorphism. We
conclude by noticing that
$$\varphi\left(A^{(g)}\right)=B^{\bigl(\psi(g)\bigr)}$$
and therefore $\Gamma _1$ and $\Gamma _2$ are equivalent.
\end{proof}
\begin{corollary}\label{cor:eqvsweakeq}
In the following cases the weak equivalence of gradings of finite
dimensional algebras implies the equivalence of gradings:
\begin{enumerate}
\item the standard gradings on twisted group algebras;
 \item strong gradings;
 \item nondegenerate gradings.
\end{enumerate}
\end{corollary}
Based on Corollary~\ref{cor:eqvsweakeq}, it is natural to ask when an
isomorphism of twisted group rings implies a graded equivalence of the
twisted group rings. It is clear that two isomorphic twisted group
rings may be not graded equivalent. A simple example for that is
the group algebras $\mathbb{C}C_4$ and $\mathbb{C}(C_2\times
C_2)$ (by $C_n$ we denote the cyclic group of order $n$). However, can this phenomenon happen for two twisted group
algebras of the same group? It turns out that, provided the group
$G$ is abelian, if $\mathbb{C}^{\sigma}G$ and $\mathbb{C}^{\rho}G$ are
isomorphic and simple, then they are graded equivalent~\cite[Theorem~18]{AljaHaileNatapov}, \cite[Proposition~2.4 (2)]{ginosargradings} (see the description of finite abelian groups of central type e.g. in~\cite[Theorem 2.15]{ElduqueKochetov}). Nonetheless,
for a non-abelian group $G$ it can happen that
$\mathbb{C}^{\sigma}G$ and $\mathbb{C}^{\rho}G$ are isomorphic
and simple, but they are not graded equivalent. Rather
complicated examples for that can be found in~\cite[\S
3.5]{ginosargradings}. However, if we relax the simplicity condition
to the graded simplicity (twisted group algebras are always graded-simple), examples even for abelian
groups can be constructed and they are much simpler than those in~\cite[\S
3.5]{ginosargradings}.
\begin{example}\label{ExampleGinosar}
Let
$$G=C_4\times C_2\times C_2=\langle x\rangle \times \langle y\rangle \times \langle z\rangle.$$

Recall that in order to define a cohomology class for a finitely generated abelian group, it is enough to determine the values of the corresponding alternating bicharacter (see Section~\ref{SubsectionCohomologyGradedDivisionAlgebras}). Define two non-cohomologous classes $[\sigma],[\rho] \in H^2(G,\mathbb{C}^\times)$
as follows:
$$[\sigma]:\ \alpha(x,y)=-1,\ \alpha(x,z)=1,\ \alpha(y,z)=1,$$
$$[\rho]:\ \beta(x,y)=1,\ \beta(x,z)=1,\ \beta(y,z)=-1.$$
Here $\alpha$ and $\beta$ are the alternating bicharacters
corresponding to $[\sigma]$ and $[\rho]$, 
respectively. 

Let
$$\mathbb{C}^{\sigma}G=\bigoplus_{i=1}^k M_{m_i}(\mathbb{C}), \quad \mathbb{C}^{\rho}G=\bigoplus_{j=1}^t M_{n_j}(\mathbb{C})$$
be the corresponding Artin--Wedderburn decompositions. Since
$[\sigma]$ and $[\rho]$ are nontrivial, we have $m_i,n_j>1$ for any
$1\leqslant i \leqslant k$, $1\leqslant j \leqslant t$. On the other hand, 
since the centers of both $\mathbb{C}^{\sigma}G$ and $\mathbb{C}^{\rho}G$
 have dimensions greater than $1$, by a simple
calculation, we get
$$\mathbb{C}^{\sigma}G\cong \mathbb{C}^{\rho}G\cong \bigoplus_{i=1}^4 M_2(\mathbb{C}).$$
However, there is a homogenous element of order $4$ in the center
of $\mathbb{C}^{\rho}G$ while there is no such homogeneous element
in the center of $\mathbb{C}^{\sigma}G$.
 Therefore, these algebras are not graded equivalent and
hence by Corollary~\ref{cor:eqvsweakeq} they are also not weakly
equivalent.
\end{example}

\section{Group-theoretical approach}
\label{SectionGroupTheoreticalApproach}

Each group grading on an algebra can be realized as a $G$-grading for many different groups $G$, however it turns out that there is one distinguished group among them \cite[Definition~1.17]{ElduqueKochetov}, \cite{PZ89}.

\begin{definition}\label{def:universal}
Let $\Gamma$ be a group grading on an algebra $A$. Suppose that $\Gamma$ admits a realization
as a $G_\Gamma$-grading for some group $G_\Gamma$. Denote by $\varkappa_\Gamma$ the corresponding embedding
$\supp \Gamma \hookrightarrow G_\Gamma$. We say that $(G_\Gamma,\varkappa_\Gamma)$ is the \textit{universal group of the grading $\Gamma$} if for any realization of $\Gamma$ as a grading by a group $G$
with $\psi \colon \supp \Gamma \hookrightarrow G$ there exists
a unique homomorphism $\varphi \colon G_\Gamma \to G$ such that the following diagram is commutative:
$$\xymatrix{ \supp \Gamma \ar[r]^\varkappa \ar[rd]^\psi & G_\Gamma \ar@{-->}[d]^\varphi \\
& G
}
$$
\end{definition}

Given a set $X$, denote by $\mathcal F(X)$ the free group with the set $X$ of free generators.
It is easy to see that if $G$ is a group and $\Gamma \colon A = \bigoplus_{g\in G} A^{(g)}$ is a grading, then $$G_\Gamma \cong \mathcal F([\supp \Gamma])/N$$ where $[\supp \Gamma]:=\lbrace [g] \mid g\in \supp \Gamma \rbrace$ and $N$ is the normal closure of the words $[g][h][gh]^{-1}$ for $g,h \in \supp \Gamma$
such that $A^{(g)}A^{(h)}\ne 0$.

The observation below is a direct consequence of the definition.

\begin{proposition}
A grading $\Gamma$ admits an infinite group turning $\Gamma$ into a connected grading if and only if $G_\Gamma$ is infinite.
\end{proposition}

In the definition above the universal group of a grading $\Gamma$ is a pair $(G_\Gamma,\varkappa_\Gamma)$.
Theorem~\ref{TheoremGivenFinPresGroupExistence} below shows, in particular, that the first component of this pair can be an arbitrary finitely presented group.
Furthermore, we can choose $\Gamma$ to be an elementary grading on a full matrix algebra. The possibility to include a subset $V$ to the support will be used later, in the proof of Theorem~\ref{TheoremFiniteRegradingImpossible}.

\begin{theorem}\label{TheoremGivenFinPresGroupExistence}
Let $F$ be a field, let $G$ be a finitely presented
group and let $V \subseteq G$ be a finite subset (possibly empty). Then for some $n\in\mathbb N$, depending only on the presentation of $G$ and the elements of $V$, there exists an elementary grading $\Gamma$ on $M_n(F)$ such that $G_\Gamma \cong G$ and $V \subseteq \supp \Gamma$.
\end{theorem}
\begin{proof}
Suppose $G \cong \mathcal F(X)/N$ where $X=\lbrace x_1, x_2, \dots, x_\ell\rbrace$ is a finite set of generators and $N$ is the normal closure of a finite set of words $w_1, \dots, w_m$.
Let $V=\lbrace g_1, \dots, g_s\rbrace \subseteq G$.
Choose $w_{m+1},\ldots, w_{m+s}\in \mathcal F(X)$
such that $g_i=\bar w_{m+i}$, $1\leqslant i \leqslant s$, where by $\bar u$ we denote the image of $u\in \mathcal F(X)$ in $G$. Suppose $w_i = y_{i1}\cdots y_{ik_i}$ where $y_{ij}\in X \cup X^{-1}$, $1\leqslant i \leqslant m+s$.

Without loss of generality, we may assume that the first $\ell_0$
generators $x_1, \ldots, x_{\ell_0}$, where $0 \leqslant \ell_0 \leqslant \ell$, do not occur among $y_{ij}$ and $y_{ij}^{-1}$,
and for each $\ell_0 < k \leqslant \ell$ there
exist $i,j$ such that either $y_{ij}=x_k$ or $y_{ij}=x_k^{-1}$.
 Denote $n=\ell_0+1+\sum_{i=1}^{m+s} k_i$.

Let $$\Gamma \colon M_n(F)=\bigoplus_{g\in G} M_n(F)^{(g)}$$ be the elementary $G$-grading
defined as follows: $$e_{r,r+1} \in M_n(F)^{(\bar y_{ij})} \text{ if } r=k_1+\dots+k_{i-1}+j,
\ 1\leqslant j \leqslant k_i,\ 1\leqslant i \leqslant m+s$$ and $$e_{r,r+1} \in M_n(F)^{(\bar x_j)}\text{ if   }r=\sum_{i=1}^{m+s} k_i+j,\ 1\leqslant j \leqslant \ell_0$$
(the corresponding elementary grading exists by Remark~\ref{RemarkElementary}).

Note that $$e_{\sum_{i=1}^{m+r-1} k_i+1, \sum_{i=1}^{m+r} k_i+1}
= \prod_{j=1}^{k_{m+r}} e_{\sum_{i=1}^{m+r-1} k_i + j,\sum_{i=1}^{m+r-1} k_i + j+1}  \in M_n(F)^{\left(\prod_{j=1}^{k_{m+r}}\bar y_{m+r,j}\right)}=M_n(F)^{(g_r)}$$
 is nonzero. Hence $g_r \in \supp \Gamma$ for each $1\leqslant r \leqslant s$
 and $V\subseteq \supp \Gamma$. 

Now we claim that $G_\Gamma \cong G$.

Suppose that $\Gamma$ is realized as a grading by a group $H$.
Then there exists an injective map $\psi \colon \supp \Gamma \hookrightarrow H$
defined by $M_n(F)^{(g)}\subseteq M_n(F)^{(\psi(g))}$.
We have \begin{equation}\label{EqPsiPartialHom} \psi(g_1 g_2)=\psi(g_1)\psi(g_2) \end{equation}
 for any $g_1, g_2 \in G$ such that $M_n(F)^{(g_1)}M_n(F)^{(g_2)}\ne 0$.
Since $M_n(F)$ is a unital algebra, we have $\psi(1_G)=1_H$.

For every $x\in X$ we have
$\bar x, \bar x^{-1}\in\supp\Gamma$. Thus elements $\psi(\bar x)$ and $\psi(\bar x^{-1})=\psi(\bar x)^{-1}$ are defined for all $x\in X$.
By induction, $$\psi(\bar y_{i1})\cdots \psi(\bar y_{ik_i})=\psi(\bar y_{i1}\cdots \bar y_{ik_i})=\psi(\bar w_i)=1_H$$
for all $1\leqslant i \leqslant m$.
Hence the elements $\psi(\bar x)$ satisfy the relations of $G$.
Therefore there exists a homomorphism $\varphi \colon G \to H$ such that $\varphi(\bar x) = \psi(\bar x)$ for each $x\in X$. Since the set $\lbrace \bar x \mid x \in X\rbrace$ generates $G$, such a homomorphism is unique.

Now we have to prove that $\varphi \bigr|_{\supp \Gamma} = \psi$.
Every element $g$ of $\supp \Gamma$ corresponds to a matrix unit
$e_{ij}$ where either $i > j$ and $e_{ij}=e_{i, i-1} e_{i-1,i-2}\cdots e_{j+1,j}$, or $i < j$ and $e_{ij}=e_{i, i+1} e_{i+1,i+2}\cdots e_{j-1,j}$,
or $i=j$ and $g=1_G$. Since for every $1\leqslant i,j \leqslant n$, such that
$|i-j|=1$, we have $e_{ij} \in M_n(F)^{(\bar x)}$ for some $x \in X \cup X^{-1}$ and $\varphi(x)=\psi(x)$ for $x \in X \cup X^{-1}$,
the induction on $|i-j|$ using~(\ref{EqPsiPartialHom}) shows that $\varphi(g)=\psi(g)$.
 Hence $G \cong G_\Gamma$.
\end{proof}

\begin{remark} For each grading $\Gamma$ one can define a category $\mathcal C_\Gamma$
where the objects are all pairs $(G,\psi)$ such that $G$ is a group and $\Gamma$ can be realized
as a $G$-grading with $\psi \colon \supp \Gamma \hookrightarrow G$ being the embedding of the support.
In this category the set of morphisms between $(G_1,\psi_1)$ and $(G_2,\psi_2)$ consists
of all group homomorphisms $f \colon G_1 \to G_2$ such that the diagram below is commutative:
$$\xymatrix{ \supp \Gamma \ar[r]^{\psi_1} \ar[rd]^{\psi_2} & G_1 \ar[d]^f \\
& G_2
}
$$
Then $(G_\Gamma,\varkappa_\Gamma)$ is the initial object of $\mathcal C_\Gamma$.
\end{remark}

\begin{definition}
Let $\Gamma_1 \colon A=\bigoplus_{g \in G} A^{(g)}$
and $\Gamma_2 \colon A=\bigoplus_{h \in H} A^{(h)}$ be two gradings
where $G$ and $H$ are groups and $A$ is an algebra.
We say that $\Gamma_2$ is \textit{coarser} than $\Gamma_1$ if
 for every $g\in G$ with $A^{(g)}\ne 0$ there exists $h\in H$
such that $A^{(g)}\subseteq A^{(h)}$. In this case $\Gamma_2$ is  called a \textit{coarsening} of $\Gamma_1$ and $\Gamma_1$ is called a \textit{refinement} of $\Gamma_2$.
Denote by $\pi_{\Gamma_1 \to \Gamma_2} \colon G_{\Gamma_1} \twoheadrightarrow G_{\Gamma_2}$
the homomorphism defined by $\pi_{\Gamma_1 \to \Gamma_2}(\varkappa_{\Gamma_1}(g))=\varkappa_{\Gamma_2}(h)$
for $g\in\supp\Gamma_1$ and $h\in\supp \Gamma_2$ such that $A^{(g)} \subseteq A^{(h)}$.
\end{definition}

\begin{notation}
For a subset $U$ of a group $G$, denote by $\diff U := \lbrace uv^{-1} \mid u,v\in U,\ u\ne v\rbrace$.
\end{notation}

\begin{lemma}\label{LemmaCoarseningSubset} 
Let $\Gamma_2 \colon A=\bigoplus_{h \in H} A^{(h)}$ be a coarsening of $\Gamma_1 \colon A=\bigoplus_{g \in G} A^{(g)}$.
Let \begin{equation*}\begin{split}W := \lbrace \varkappa_{\Gamma_1}(g_1) \varkappa_{\Gamma_1}(g_2)^{-1} \mid
g_1,g_2 \in \supp \Gamma_1\text{ such that } \\ A^{(g_1)}\oplus A^{(g_2)} \subseteq A^{(h)}\text{ for some }h\in \supp \Gamma_2\rbrace.\end{split}\end{equation*}
Denote by $Q \triangleleft  G_{\Gamma_1}$ the normal closure of $W$.
 Then $\ker \pi_{\Gamma_1 \to \Gamma_2} = Q$. In addition, $Q \cap \diff\varkappa_{\Gamma_1}(\supp \Gamma_1)= W$.
\end{lemma}
\begin{proof}
Obviously, $Q\subseteq \ker \pi_{\Gamma_1 \to \Gamma_2}$.

Let $\pi_{\Gamma_1} \colon \mathcal F([\supp \Gamma_1]) \twoheadrightarrow G_{\Gamma_1}$
and $\pi_{\Gamma_2} \colon \mathcal F([\supp \Gamma_2]) \twoheadrightarrow G_{\Gamma_2}$
be the natural surjective homomorphisms. Denote by $\varphi \colon \mathcal F([\supp \Gamma_1]) \twoheadrightarrow \mathcal F([\supp \Gamma_2])$ the surjective homomorphism defined by $\varphi(g) = h$ for $g\in\supp\Gamma_1$ and $h\in\supp \Gamma_2$ such that $A^{(g)} \subseteq A^{(h)}$.
Then the following diagram is commutative:
\begin{equation}\label{EqCommDiagGamma1Gamma2}\xymatrix{ \mathcal F([\supp \Gamma_1]) \ar[r]^(0.65){\pi_{\Gamma_1}} \ar[d]^{\varphi} & G_{\Gamma_1} \ar[d]^{\pi_{\Gamma_1 \to \Gamma_2}}\\ \mathcal F([\supp \Gamma_2])
\ar[r]^(0.65){\pi_{\Gamma_2}}  & G_{\Gamma_2}}\end{equation}

Note that $\ker \varphi$ coincides with the normal closure in $\mathcal F([\supp \Gamma_1])$
of all elements $[g_1][g_2]^{-1}$ where $g_1,g_2 \in \supp \Gamma_1$ and $A^{(g_1)}\oplus A^{(g_2)} \subseteq A^{(h)}$ for some $h\in \supp \Gamma_2$. Hence $\pi_{\Gamma_1}(\ker\varphi) = Q$.

Suppose $\pi_{\Gamma_1 \to \Gamma_2} \pi_{\Gamma_1}(w) = 1_{G_{\Gamma_2}}$ for some $w\in \mathcal F([\supp \Gamma_1])$. Then~(\ref{EqCommDiagGamma1Gamma2}) implies $\varphi(w)\in \ker \pi_{\Gamma_2}$.
Therefore $\varphi(w)$ belongs to the normal closure of the words
$[h_1][h_2] [h_1 h_2]^{-1}$ for $h_1, h_2 \in \supp \Gamma_2$ with $A^{(h_1)}A^{(h_2)}\ne 0$.
However the last inclusion holds if and only if $A^{(g_1)}A^{(g_2)}\ne 0$
for some $g_1, g_2 \in \supp \Gamma_1$ such that $A^{(g_1)}\subseteq A^{(h_1)}$, $ A^{(g_2)}\subseteq A^{(h_2)}$.
Hence we can rewrite $w=w_0 w_1$ where $w_0 \in \ker\varphi$ and $w_1 \in \ker \pi_{\Gamma_1}$.
In particular, $\pi_{\Gamma_1}(w) = \pi_{\Gamma_1}(w_0)\in Q$. Since $\pi_{\Gamma_1}$
is surjective, we get $\ker \pi_{\Gamma_1 \to \Gamma_2} = Q$.
Together with the obvious equality
 $\ker \pi_{\Gamma_1 \to \Gamma_2} \cap \diff\varkappa_{\Gamma_1}(\supp \Gamma_1)= W$ this implies
 $Q \cap \diff\varkappa_{\Gamma_1}(\supp \Gamma_1)= W$.
\end{proof}

\begin{lemma}\label{LemmaSubsetCoarseningExists} Let $\Gamma_1 \colon A=\bigoplus_{g \in G} A^{(g)}$ be a grading by a group $G$.
Then for each subset $W \subseteq \diff\varkappa_{\Gamma_1}(\supp \Gamma_1)$
there exists a coarsening $\Gamma_2 \colon A=\bigoplus_{h \in H} A^{(h)}$ of $\Gamma_1$
such that $\ker\pi_{\Gamma_1 \to \Gamma_2} = Q$
where $Q$ is the normal closure of $W$ in $G_{\Gamma_1}$.
\end{lemma}
\begin{proof}
Let $\pi_{\Gamma_1} \colon \mathcal F([\supp \Gamma_1]) \twoheadrightarrow G_{\Gamma_1}$ and $\pi \colon G_{\Gamma_1} \twoheadrightarrow G_{\Gamma_1}/Q$
be the natural surjective homomorphisms. Consider the grading $\Gamma_2 \colon A=\bigoplus_{u \in G_{\Gamma_1}/Q} A^{(u)}$ where $A^{(u)} := \bigoplus_{\substack{g\in \supp \Gamma_1,\\ \pi(\varkappa_{\Gamma_1}(g))=u}} A^{(g)}$. 
We claim that $G_{\Gamma_1}/Q$ is the universal group of the grading $\Gamma_2$.
If $\Gamma_2$
can be realized as a grading by a group $H$ 
and $\psi \colon \supp\Gamma_2\hookrightarrow H$ is the corresponding embedding of the support,
then there exists a unique homomorphism $\varphi \colon \mathcal F([\supp \Gamma_1]) \to H$
such that $\varphi([g])=\psi(\pi(\varkappa_{\Gamma_1}(g)))$
for all $g\in \supp \Gamma_1$.
Note that $\ker (\pi\pi_{\Gamma_1})$ is the normal closure in $\mathcal F([\supp \Gamma_1])$ of:
\begin{enumerate}
\item the words $[g][h][gh]^{-1}$ for all  $g,h \in \supp \Gamma_1$
such that $A^{(g)}A^{(h)}\ne 0$;
\item the words $[g_1] [g_2]^{-1}$ for all $\varkappa_{\Gamma_1}(g_1) \varkappa_{\Gamma_1}(g_2)^{-1} \in W$.
\end{enumerate}
Hence $\ker (\pi\pi_{\Gamma_1}) \subseteq \ker \varphi$ and
there exists a homomorphism $\bar\varphi \colon G_{\Gamma_1}/Q \to H$ such that
$\bar \varphi\pi\pi_{\Gamma_1} = \varphi$. In particular, $\bar\varphi(u)=\psi(u)$ for all $u\in \supp \Gamma_2$.
Since $G_{\Gamma_1}/Q$ is generated by $\supp \Gamma_2$, the homomorphism $G_{\Gamma_1}/Q \to H$
with this property is unique, $G_{\Gamma_1}/Q$ is the universal group of the grading $\Gamma_2$
and $\pi_{\Gamma_1 \to \Gamma_2}$ can be identified with $\pi$.
\end{proof}
\begin{remark}
Note that the inclusion $W \subseteq Q \cap \diff\varkappa_{\Gamma_1}(\supp \Gamma_1)$
in Lemma~\ref{LemmaSubsetCoarseningExists}
can be strict.
\end{remark}

\begin{definition}\label{def:resfinite}
Let $G$ be a group and let $W$ be a subset of $G$. We say that $G$ is \textit{residually finite with respect to $W$} if there exists a normal subgroup $N \triangleleft G$ of finite index
such that $W \cap N = \varnothing$.
We say that $G$ is \textit{hereditarily residually finite with respect to $W$} if
for the normal closure $N_1 \triangleleft G$ of any subset of $W$
there exists a normal subgroup $N \triangleleft G$ of finite index
such that $N_1 \subseteq N$ and $W \cap N = W \cap N_1$.
\end{definition}

Theorem~\ref{TheoremFinGradingsToGroupsTransition}
below shows that the problem of whether a grading and its coarsenings can be regraded by a finite group can be viewed completely group theoretically.

\begin{theorem}\label{TheoremFinGradingsToGroupsTransition} Let $G$ be a group, $A$ be an algebra, and let $\Gamma \colon A = \bigoplus_{g\in G} A^{(g)}$ be a $G$-grading on $A$.
Then \begin{enumerate}
\item $\Gamma$ is weakly equivalent to a grading by a finite group if and only if
$G_\Gamma$ is residually finite with respect to $\diff\varkappa_\Gamma(\supp \Gamma)$;
\item $\Gamma$ and all its coarsenings are weakly equivalent to a grading by a finite group if and only if
$G_\Gamma$ is hereditarily residually finite with respect to $\diff\varkappa_\Gamma(\supp \Gamma)$.
\end{enumerate}
\end{theorem}
\begin{proof}
The grading $\Gamma$ can be realized by any factor group of $G_\Gamma$ that does not glue the elements of the support, i.e. distinct elements of the support have distinct images in that factor group. Then if $G_\Gamma$ is residually finite with respect to $\diff\varkappa_\Gamma(\supp \Gamma)$,
there exists a finite factor group with this property. Conversely, if $\Gamma$ admits a realization
as a grading by a finite group $G$, then the subgroup of $G$ generated by the support
is a finite factor group of $G_\Gamma$ that does not glue the elements of the support.
Hence $G_\Gamma$ is residually finite with respect to $\diff\varkappa_\Gamma(\supp \Gamma)$ and the first part of the theorem is proved.

Suppose $G_\Gamma$ is hereditarily residually finite with respect to $\diff\varkappa_\Gamma(\supp \Gamma)$.
By Lemma~\ref{LemmaCoarseningSubset}, for any coarsening $\Gamma_1$ of $\Gamma$
there exists a normal subgroup $Q$ which is the normal closure of
\begin{equation*}\begin{split}Q \cap
\diff\varkappa_\Gamma(\supp \Gamma) = \lbrace \varkappa_{\Gamma}(g_1) \varkappa_{\Gamma}(g_2)^{-1} \mid
g_1,g_2 \in \supp \Gamma\text{ such that } \\ A^{(g_1)}\oplus A^{(g_2)} \subseteq A^{(h)}\text{ for some }h\in \supp \Gamma_1\rbrace,\end{split}\end{equation*}
such that $G_{\Gamma_1} \cong G_{\Gamma}/Q$. Since $G_\Gamma$ is hereditarily residually finite with respect to $\diff\varkappa_\Gamma(\supp \Gamma)$, there exists a normal subgroup $N \triangleleft G_{\Gamma}$ of finite index such that $Q \subseteq N$ and
\begin{equation}\label{EqNQDiffKappa} N \cap \diff\varkappa_\Gamma(\supp \Gamma) = Q \cap \diff\varkappa_\Gamma(\supp \Gamma).\end{equation}
Let $\bar N=\pi_{\Gamma \to \Gamma_1}(N)$. Suppose $\varkappa_{\Gamma_1}(h_1)\varkappa_{\Gamma_1}(h_2)^{-1} \in \bar N$ for some $h_1, h_2 \in \supp \Gamma_1$. Using the isomorphism $G_\Gamma/N \cong G_{\Gamma_1}/\bar N$ we get $\varkappa_{\Gamma}(g_1)\varkappa_{\Gamma}(g_2)^{-1} \in N$ for all $g_1, g_2 \in \supp \Gamma$ such that $A^{(g_1)} \subseteq A^{(h_1)}$ and $A^{(g_2)} \subseteq A^{(h_2)}$.
Now~(\ref{EqNQDiffKappa})
implies $\varkappa_{\Gamma_1}(g_1)\varkappa_{\Gamma_1}(g_2)^{-1} \in Q$,
$\varkappa_{\Gamma_1}(h_1)=\varkappa_{\Gamma_1}(h_2)$ and $h_1 = h_2$. Hence $\bar N$ does not glue the elements of the support of $\Gamma_1$. Since $|G_{\Gamma_1}/\bar N| = |G_\Gamma/N|< +\infty$,
the grading $\Gamma_1$ admits the finite grading group $G_{\Gamma_1}/\bar N$.

Suppose that every coarsening $\Gamma_1$ of $\Gamma$ admits a finite grading group. We claim that
$G_\Gamma$ is hereditarily residually finite with respect to $\diff\varkappa_\Gamma(\supp \Gamma)$.
Indeed, let $N_1 \triangleleft G_\Gamma$ be a normal closure of a subset of
$\diff\varkappa_\Gamma(\supp \Gamma)$. By Lemma~\ref{LemmaSubsetCoarseningExists},
there exists a grading $\Gamma_1$ such that $G_{\Gamma_1} \cong G_{\Gamma}/N_1$.
Since $\Gamma_1$ admits a finite grading group, there exists a normal subgroup $Q \triangleleft G_{\Gamma_1}$ of finite index such that $\varkappa_{\Gamma_1}(h_1) \varkappa_{\Gamma_1}(h_2)^{-1} \notin Q$
for all $h_1, h_2 \in \supp \Gamma_1$, $h_1 \ne h_2$. Let $N := \pi_{\Gamma \to \Gamma_1}^{-1}(Q)$.
Then $|G_\Gamma/N| = |G_{\Gamma_1}/ Q| < +\infty$, $N \supseteq \ker\pi_{\Gamma \to \Gamma_1}=N_1$, and
$$N \cap \diff\varkappa_\Gamma(\supp \Gamma) \supseteq N_1 \cap \diff\varkappa_\Gamma(\supp \Gamma).$$
Suppose $\varkappa_{\Gamma}(g_1) \varkappa_{\Gamma}(g_2)^{-1} \in N$ for some $g_1, g_2 \in \supp \Gamma$.
Take $h_1, h_2 \in \supp \Gamma_1$ such that $A^{(g_1)} \subseteq A^{(h_1)}$ and $A^{(g_2)} \subseteq A^{(h_2)}$. Then $\varkappa_{\Gamma_1}(h_1) \varkappa_{\Gamma_1}(h_2)^{-1} \in Q$
and $h_1 = h_2$. Thus $\varkappa_{\Gamma}(g_1) \varkappa_{\Gamma}(g_2)^{-1} \in N_1$
and $$N \cap \diff\varkappa_\Gamma(\supp \Gamma) = N_1 \cap \diff\varkappa_\Gamma(\supp \Gamma).$$
As a consequence, $G_\Gamma$ is hereditarily residually finite with respect to $\diff\varkappa_\Gamma(\supp \Gamma)$.
\end{proof}

The proposition below could be obtained as a consequence of Theorem~\ref{TheoremFinGradingsToGroupsTransition}, however we prefer to give a separate proof.

\begin{proposition}\label{prop:abelian}
Let $\Gamma \colon A=\bigoplus_{g \in G} A^{(g)}$ be a grading of a finite dimensional
algebra $A$ by an abelian group $G$. Then $\Gamma$ is weakly equivalent to a grading by a finite group.
\end{proposition}
\begin{proof}
We can replace $G$ with its subgroup generated by $\supp \Gamma$. Since $\supp \Gamma$ is finite,
without loss of generality we may assume that $G$ is a finitely generated abelian group.
Then $G$ is a direct product of free and primary cyclic groups. Replacing free cyclic groups
with cyclic groups of a large enough order (see Example~\ref{ExampleAbelianRegrading} below), we get a finite grading group.
\end{proof}

\begin{example}\label{ExampleAbelianRegrading}
Let $n\in \mathbb N$ and let $\Gamma$ be the elementary $\mathbb Z$-grading on $M_n(F)$
defined by the $n$-tuple $(1,2,\ldots,n)$, i.e. $e_{ij}\in M_n(F)^{(i-j)}$.
Then $$\supp \Gamma = \lbrace -(n-1), -(n-2), \ldots, -1, 0, 1, 2, \ldots, n-1 \rbrace$$
and $\Gamma$ is equivalent to the elementary $\mathbb Z/(2n\mathbb Z)$-grading
defined by the $n$-tuple $(\bar 1,\bar 2,\ldots,\bar n)$, i.e. $e_{ij}\in M_n(F)^{(\overline{i-j})}$.
\end{example}

Consider the grading $\Gamma_0$ on $M_n(F)$ by the free group $\mathcal F(x_1, \dots, x_{n-1})$
such that $e_{r,r+1} \in M_n(F)^{(x_r)}$ for $1\leqslant r \leqslant n-1$,
i.e. defined by the $n$-tuple $(x_1 x_2 \cdots x_{n-1}, x_2 \cdots x_{n-1}, \ldots, x_{n-1}, 1)$.
Note that the neutral element component of $\Gamma_0$ is the linear span of matrix units $e_{ii}$, $1\leqslant i \leqslant n$, and is $n$-dimensional, and all the rest components are $1$-dimensional. Since for each elementary grading the diagonal matrix units $e_{ii}$ belong to the neutral element component
of the grading and all matrix units are homogeneous, every elementary grading on $M_n(F)$
is a coarsening of $\Gamma_0$. Since $\mathcal F(x_1, \dots, x_{n-1})$ is free and all its free generators
belong to $\supp \Gamma_0$, $G_{\Gamma_0} \cong \mathcal F(x_1, \dots, x_{n-1})$.
Note also that $$\supp \Gamma_0 = \lbrace x_i x_{i+1}\cdots x_j  \mid 1\leqslant i \leqslant j \leqslant n-1 \rbrace\cup
\lbrace x_j^{-1} x_{j-1}^{-1}\cdots x_i^{-1}  \mid 1\leqslant i \leqslant j \leqslant n-1 \rbrace \cup \lbrace 1\rbrace.$$
Thus by Theorem~\ref{TheoremFinGradingsToGroupsTransition} Problem~\ref{ProblemMinNElAGradNonFin}
is equivalent to Problem~\ref{ProblemMinFnHRFRx1x2x3} below:
\begin{problem}\label{ProblemMinFnHRFRx1x2x3} 
Determine
the set $\Omega$ of  
the numbers $n\in\mathbb N$ such that the group $\mathcal F(x_1, \dots, x_{n-1})$ is hereditarily residually finite with respect to $\diff W_n$
where \begin{equation}\label{EqWnFreeGroup}W_n=\lbrace x_i x_{i+1}\cdots x_j  \mid 1\leqslant i \leqslant j \leqslant n-1 \rbrace\cup
\lbrace x_j^{-1} x_{j-1}^{-1}\cdots x_i^{-1}  \mid 1\leqslant i \leqslant j \leqslant n-1 \rbrace \cup \lbrace 1\rbrace.\end{equation}
\end{problem}

\section{Regrading full matrix algebras by finite groups}
\label{SectionRegradingMatrixFinite}

In this section we prove Theorem~\ref{TheoremRegradeElementaryOmega}
that deals with Problems~\ref{ProblemMinNElAGradNonFin} and~\ref{ProblemMinFnHRFRx1x2x3}.

In the lemma below we use the idea of~\cite[Section 4]{ClaseJespersDelRio}
and show, in particular, that all semigroup regradings of elementary
group gradings on $M_n(F)$ can be reduced to group regradings.

\begin{lemma}\label{LemmaSemigroupMatrixElementaryGroup}
Let $F$ be a field and let $n\in\mathbb N$.
If $\Gamma \colon M_n(F)=\bigoplus_{t\in T} M_n(F)^{(t)}$
is a grading on $M_n(F)$ by a semigroup $T$ such that all $e_{ij}$ are homogeneous elements and there exists an element $e \in T$
such that all $e_{ii}\in M_n(F)^{(e)}$, then $e^2=e$
and $\supp \Gamma \subseteq U(eTe)$ where
$U(eTe)$ is the group of invertible elements of the monoid $eTe$.
\end{lemma}
\begin{proof}
$e_{11}^2=e_{11}$ implies $e^2=e$. Since the identity matrix belongs to
$M_n(F)^{(e)}$, we obtain $\supp \Gamma \subseteq eTe$.
Now $e_{ij}e_{ji}=e_{ii}$ implies $\supp \Gamma \subseteq U(eTe)$.
\end{proof}

The lemma below and its corollary show that if for some $m\in\mathbb N$
we have $m\notin \Omega$ (see the definition of $\Omega$ in Problem~\ref{ProblemMinNElAGradNonFin}), then $n\notin \Omega$ for all $n\geqslant m$.

\begin{lemma}\label{LemmaGradedSubalgebra}
Let $\Gamma \colon A = \bigoplus_{t\in T} A^{(t)}$
be a grading by a (semi)group $T$ on an algebra $A$ over a field $F$
and let $B$ be a graded subalgebra. Suppose that the grading on $B$
cannot be regraded by a finite (semi)group. Then $\Gamma$ cannot be regraded
by a finite (semi)group either.
\end{lemma}
\begin{proof}
Each regrading on $\Gamma$ induces a regrading of the grading on $B$.
Therefore, if it were possible to regrade $\Gamma$ by a finite (semi)group,
the same would be possible for the grading on $B$. However, the latter is impossible.
\end{proof}

\begin{corollary}\label{CorollaryMatrixElemGreaterSize}
If for some $m\in\mathbb N$ and a group $G$ there exists an elementary $G$-grading on $M_m(F)$ that is not weakly equivalent to a grading by a finite (semi)group,
then an elementary $G$-grading with this property exists on $M_n(F)$ for every $n\geqslant m$.
\end{corollary}
\begin{proof}
Suppose this elementary $G$-grading on $M_m(F)$ can be realized by an $m$-tuple
$(g_1, g_2, \dots, g_m)$. Consider the elementary $G$-grading $\Gamma$ on $M_n(F)$
defined by the $n$-tuple $(g_1, g_2, \dots, g_m,\underbrace{g_m, \dots, g_m}_{n-m})$. The algebra $M_m(F)$ becomes a graded subalgebra of $M_n(F)$ (with a different identity element).
Therefore, if $\Gamma$ were weakly equivalent to a grading by a finite (semi)group,
then it would be possible to reindex the graded components of $M_n(F)$
by elements of a finite group and the original $G$-grading on $M_m(F)$ had this property too. Hence $\Gamma$ is not weakly equivalent to any grading by a finite (semi)group.
\end{proof}
Since Problems~\ref{ProblemMinNElAGradNonFin} and~\ref{ProblemMinFnHRFRx1x2x3}
are equivalent, we immediately get
\begin{corollary}
If $\mathcal F(x_1, \dots, x_{m-1})$ is not hereditarily residually finite with respect to $\diff W_m$ for some $m\in\mathbb N$, then
$\mathcal F(x_1, \dots, x_{n-1})$ is not hereditarily residually finite with respect to $\diff W_n$ for all $n \geqslant m$. (See the definition of $W_n$ in~(\ref{EqWnFreeGroup}).)
\end{corollary}

Recall that a group is \textit{residually finite} if the intersection of its normal subgroups of finite index is trivial.

\begin{theorem}\label{TheoremFiniteRegradingImpossible} Let $F$ be a field and let $G$ be a finitely presented
group which is not residually finite. (For example, $G$ is a finitely presented infinite simple group, see~\cite{Higman}.)
Then there exists an elementary $G$-grading on a full matrix algebra which is not weakly equivalent to any $H$-grading for any finite (semi)group $H$.
\end{theorem}
\begin{proof}
Let $g_0\ne 1_G$ be an element that belongs to the intersection of
all normal subgroups of $G$ of finite index. (In particular, if $G$ is simple, we take 
an arbitrary element $g_0\ne 1_G$.)

By Theorem~\ref{TheoremGivenFinPresGroupExistence},
there exists an elementary $G$-grading $\Gamma$ on $M_n(F)$ for some $n\in\mathbb N$
such that $G_\Gamma \cong G$ and $g_0 \in \supp G$.

Suppose that $\Gamma$ is weakly equivalent to a grading by a finite semigroup $H$ and $\psi \colon \supp \Gamma \hookrightarrow H$ is the corresponding embedding of the support. Then $\Gamma$ can be regraded by $H$
and there exists $e\in H$ such that all $e_{ii} \in M^{(e)}_n(F)$.
Since all elements $e_{ij}$ are homogeneous,
by Lemma~\ref{LemmaSemigroupMatrixElementaryGroup} we may assume that $H$ is a group. Since $G_\Gamma \cong G$, there exists a unique homomorphism $\varphi \colon G \to H$ such that $\varphi|_{\supp \Gamma} = \psi$.
However $H$ is finite, $\ker \varphi$ is of finite index, and  therefore $\varphi(g_0)=1_H$. Since $g_0 \ne 1_G$ and $g_0, 1_G \in \supp \Gamma$, this $H$-grading cannot be weakly equivalent to $\Gamma$.
\end{proof}

\begin{corollary}\label{CorollaryFiniteRegradingImpossible349} If $n \geqslant 349$, then $M_n(F)$ admits a grading by an infinite group that is not weakly equivalent to any grading by a finite (semi)group.
\end{corollary}
\begin{proof}
Consider Thompson's finitely presented infinite simple group $G_{2,1}$ (see e.g. \cite[Section~8]{Higman}).
We can take $g_0$ to be any of its generators which are all anyway in the support
of the grading constructed in Theorem~\ref{TheoremGivenFinPresGroupExistence} for $G=G_{2,1}$.
Therefore it is sufficient to apply Theorem~\ref{TheoremGivenFinPresGroupExistence} with $V = \varnothing$.
Summing up the lengths of the
defining relators, we obtain that $n$ in the proof
of Theorem~\ref{TheoremGivenFinPresGroupExistence}
equals $349$. Now we apply Corollary~\ref{CorollaryMatrixElemGreaterSize}.
\end{proof}

Corollary~\ref{CorollaryFiniteRegradingImpossible349} implies
the upper bound in Theorem~\ref{TheoremRegradeElementaryOmega}.

Recall that an algebra $M(\gamma, \sigma)$ (see the definition in Section~\ref{SubsectionElementaryGradedSimpleAlgebras}), where $\gamma$ is an $n$-tuple of group elements, contains a graded subalgebra which is
graded isomorphic to $M_n(F)$ with the elementary grading determined by $\gamma$, namely, the subalgebra that is the linear span of $e_{ij}\otimes u_{1_H}$, $1\leqslant i,j \leqslant n$. Therefore, using Corollary~\ref{CorollaryFiniteRegradingImpossible349}
and Lemma~\ref{LemmaGradedSubalgebra}, one obtain that, for every
$n\geqslant 349$, every finite subgroup $H \subseteq G_{2,1}$, and every $\sigma \in Z^2(H, F^\times)$, there exists an $n$-tuple $\gamma$ of elements of  $G_{2,1}$ such that the standard grading on $M(\gamma, \sigma)$ is not weakly equivalent to a grading by a finite (semi)group.

Now we present a class of elementary gradings that are weakly equivalent to gradings by finite groups. Namely, consider elementary gradings where
distinct non-diagonal matrix units $e_{ij}$ belong to distinct homogeneous components corresponding to group elements $g\ne 1$.

\begin{theorem}\label{TheoremFiniteRegradingAllDifferent} Let $G$ be a group, let $F$ be a field, let $n\in\mathbb N$, and let $(g_1, \dots, g_n)$ be an $n$-tuple of elements of $G$ such that $g_i g_j^{-1} = g_k g_\ell^{-1}$ if and only if either
$\left\lbrace \begin{smallmatrix} i=k, \\ j=\ell \end{smallmatrix}\right.$ or 
$\left\lbrace \begin{smallmatrix} i=j, \\ k=\ell. \end{smallmatrix}\right.$
 Then the elementary grading on $M_n(F)$
defined by $(g_1, \dots, g_n)$ is weakly equivalent to the elementary $S_{n+1}$-grading defined by $(\gamma_1,\dots,\gamma_n)$ where $S_{n+1}$ is the symmetric group acting on $\lbrace 1,2,\dots,n+1\rbrace$ and $\gamma_i = (1, i+1) \in S_{n+1}$ is the transposition switching $1$ and $i+1$.
The same grading on $M_n(F)$ is weakly equivalent to the elementary $\mathbb Z/(2^{n+1}\mathbb Z)$-grading defined by $(\bar 2, \bar 2^2, \bar 2^3, \dots, \bar 2^n)$.
\end{theorem}
\begin{proof} In order to prove the first part of the theorem, it suffices to prove that $\gamma_i \gamma_j^{-1} = \gamma_k \gamma_\ell^{-1}$ if and only if either $\left\lbrace \begin{smallmatrix} i=k, \\ j=\ell \end{smallmatrix}\right.$ or 
$\left\lbrace \begin{smallmatrix} i=j, \\ k=\ell. \end{smallmatrix}\right.$ However, if $i=j$, then $\gamma_i \gamma_j^{-1}$ is the identity permutation and if $i\ne j$, then $\gamma_i \gamma_j^{-1} = (1, j+1, i+1)$ (a $3$-cycle).

In order to prove the second part of the theorem, we 
notice that $\bar 2^i-\bar 2^j=\bar 2^k-\bar 2^\ell$ if and only if $2^i-2^j=2^k-2^\ell$ if and only if either $\left\lbrace \begin{smallmatrix} i=k, \\ j=\ell \end{smallmatrix}\right.$ or 
$\left\lbrace \begin{smallmatrix} i=j, \\ k=\ell. \end{smallmatrix}\right.$ Indeed, dividing the equality $2^i-2^j=2^k-2^\ell$ by $2^{\min(i,j,k,\ell)}$, we see that at least two of the numbers $i,j,k,\ell$ must coincide
which implies the assertion claimed.
\end{proof}

In Theorem~\ref{TheoremFiniteRegradingImpossible} we have constructed
an elementary $G$-grading on $M_n(F)$ that is not weakly equivalent to a grading by a finite group.
However, this grading is a coarsening of the elementary grading by the free group
$\mathcal F(z_1, \dots, z_{n-1})$ that corresponds to the $n$-tuple $(1,z_1,\dots, z_{n-1})$
(this grading was considered in~\cite[Proposition~4.11]{CibilsRedondoSolotar}, \cite[Lemma~4.5]{ginosargradings})
and which is by Theorem~\ref{TheoremFiniteRegradingAllDifferent}
weakly equivalent to a grading by a finite group. In other words, there exist gradings that
can be regraded by a finite group, but some of their coarsenings cannot.

\begin{theorem}\label{TheoremFiniteRegradingSmallMatrix}
Let $\Gamma$ be an elementary $G$-grading on the full matrix algebra $M_n(F)$
where $n \leqslant 3$, $F$ is a field, and $G$ is a group.
Then $\Gamma$ is weakly equivalent to a grading by a finite group.
\end{theorem}
\begin{remark}
If $F$ is algebraically closed,
then the theorem holds for all gradings on $M_n(F)$, $n \leqslant 3$, not necessarily elementary ones.
Indeed, by Theorem~\ref{TheoremBahturinZaicevSeghal} any grading
on $M_n(F)$ is isomorphic to the standard grading on $M(\gamma, \sigma)$ for some $\gamma$ and $\sigma$. Comparing the dimensions, we obtain that either $M(\gamma, \sigma)$ is a twisted group algebra of a finite group,
i.e. there is nothing to prove, or $\sigma$ is a cocycle of the trivial group and the grading $\Gamma$ is isomorphic to an elementary one.
\end{remark}
\begin{proof}[Proof of Theorem~\ref{TheoremFiniteRegradingSmallMatrix}]
If $n=1$, then $M_n(F)=F$ and it can be regraded by the trivial group.
Therefore, we may assume that $n=2,3$.
Recall that $e_{ij} \in M_n(F)^{(g_i g_j^{-1})}$
where $(g_1, \dots, g_n)$ is the $n$-tuple defining the elementary grading.
Consider the $n\times n$ matrix $A$ where the $(i,j)$th entry is $g_i g_j ^{-1}$.
 It is clear that this matrix completely determines the grading.
Theorem~\ref{TheoremWeakEquivGradedSimple} implies
that if two elementary gradings $\Gamma_1$ and $\Gamma_2$ have matrices $A$ and $B$ such that two entries in $A$ coincide if and only if the corresponding
entries in $B$ coincide, then $\Gamma_1$ and $\Gamma_2$ are weakly equivalent.

In the case $n=2$, we get $A=\left(\begin{smallmatrix}
1 & g \\
g^{-1} & 1
\end{smallmatrix}\right)$ where $1=1_G$. Here we have two cases: $g=g^{-1}$
and $g\ne g^{-1}$. In the case $g=g^{-1}$ we can regrade $M_2(F)$
by $\mathbb Z/2\mathbb Z$ (the elementary grading is defined by the couple $(\bar 0,\bar 1)$). In the case $g\ne g^{-1}$
we can regrade $M_2(F)$ by $\mathbb Z/3\mathbb Z$ (the elementary grading is again defined by the couple $(\bar 0,\bar 1)$).

Consider now the case $n=3$.
The matrix with the entries $g_ig_j^{-1}$
is of the form
$A=\left(\begin{smallmatrix}
1      & g      & gh \\
g^{-1} & 1      & h \\
h^{-1}g^{-1} & h^{-1} & 1
\end{smallmatrix}\right)$ and the same grading can be defined by the triple $(gh,h,1)$.
If $g\in \lbrace 1, h, h^{-1}, gh, h^{-1}g^{-1} \rbrace$
or $h\in \lbrace 1, g, g^{-1}, gh, h^{-1}g^{-1} \rbrace$, then the subgroup of $G$ generated by $g$ and $h$ is abelian and by Proposition~\ref{prop:abelian} the algebra $M_3(F)$ can be regraded by a finite group.
Therefore below we assume that $1, g, h, gh$
are pairwise distinct and the condition $\lbrace g,h,gh\rbrace \cap \lbrace g^{-1},h^{-1},h^{-1}g^{-1}\rbrace
\ne \varnothing$ is true only if some of the equalities $g=g^{-1}$, $h=h^{-1}$, $gh=h^{-1}g^{-1}$ hold, in other words, only if
at least one of the elements $g$, $h$, $gh$ is of order $2$.

For three possible equalities above we have $2^3=8$ different cases depending on whether they hold or not.

\begin{enumerate}
\item $\lbrace g,h,gh\rbrace \cap \lbrace g^{-1},h^{-1},h^{-1}g^{-1}\rbrace
= \varnothing$.
Here we can apply Theorem~\ref{TheoremFiniteRegradingAllDifferent}
since all the entries, except the diagonal ones, are different.

\item $gh = h^{-1}g^{-1}$, but $\lbrace g,h\rbrace \cap \lbrace g^{-1},h^{-1}\rbrace
= \varnothing$.
Here $\Gamma$ is weakly equivalent to the elementary $\mathbb Z/6\mathbb Z$-grading defined by $g\mapsto \bar 1$ and $h\mapsto \bar 2$,
i.e. by the triple $(\bar 3, \bar 2, \bar 0)$.

\item $gh = h^{-1}g^{-1}$, $g=g^{-1}$, but $h\ne h^{-1}$.
Here $\Gamma$ is weakly equivalent to the elementary $S_3$-grading defined by $g\mapsto (12)$ and $h\mapsto (123)$.

\item $gh = h^{-1}g^{-1}$, $g\ne g^{-1}$, $h = h^{-1}$. Here $\Gamma$ is weakly equivalent to the elementary $S_3$-grading defined by $g\mapsto (123)$ and $h\mapsto (12)$.

\item $gh = h^{-1}g^{-1}$, $g=g^{-1}$, $h=h^{-1}$. Here $\Gamma$ is weakly equivalent to the elementary $(\mathbb Z/2\mathbb Z) \times (\mathbb Z/2\mathbb Z)$-grading
defined by $g\mapsto(\bar 0, \bar 1)$ and $h\mapsto(\bar 1, \bar 0)$.
(Since in this case $g$ and $h$ commute, we could have used Proposition~\ref{prop:abelian} instead.)

\item $gh \ne h^{-1}g^{-1}$, $g=g^{-1}$, $h\ne h^{-1}$.
Here $\Gamma$ is weakly equivalent to the elementary
$\mathbb Z/6\mathbb Z$-grading defined by $g\mapsto \bar 3$ and $h\mapsto \bar 1$.

\item $gh \ne h^{-1}g^{-1}$, $g\ne g^{-1}$, $h=h^{-1}$.
This case is treated analogously.

\item $gh \ne h^{-1}g^{-1}$, $g=g^{-1}$, $h=h^{-1}$. Here $\Gamma$ is weakly equivalent to the elementary $S_3$-grading defined by $g\mapsto (12)$, $h \mapsto (13)$.
\end{enumerate}

All the cases have been considered and $\Gamma$ is weakly equivalent to an elementary grading by a finite group.
\end{proof}

\begin{remark}
The elementary $S_3$-grading on $M_3(F)$ defined above by $g\mapsto (12)$, $h \mapsto (13)$ is not weakly equivalent to any of the gradings by abelian groups since $(12)(13)\ne (13)(12)$.
\end{remark}

\begin{proof}[Proof of Theorem~\ref{TheoremRegradeElementaryOmega}]
Theorem~\ref{TheoremRegradeElementaryOmega} immediately follows from Corollaries~\ref{CorollaryMatrixElemGreaterSize}, \ref{CorollaryFiniteRegradingImpossible349} and Theorem~\ref{TheoremFiniteRegradingSmallMatrix}.
\end{proof}

\section*{Acknowledgements}

The authors are grateful to Yuval Ginosar for his help in the construction of Example~\ref{ExampleGinosar}. In addition,
the authors appreciate the referee for carefully reading the manuscript
and providing a list of misprints and suggestions.

\end{document}